\documentclass[preprint,12pt]{elsarticle}

\usepackage{amssymb}
\usepackage{amsfonts}
\usepackage{amsmath}
\usepackage{amsthm}

\newtheorem{Thm}{Theorem}
\newtheorem{Lem}[Thm]{Lemma}
\newtheorem{Prop}[Thm]{Proposition}
\newtheorem{Cor}[Thm]{Corollary}

\theoremstyle{definition}
\newtheorem{Def}[Thm]{Definition}

\theoremstyle{remark}
\newtheorem{Rem}[Thm]{Remark}

\numberwithin{equation}{section}

\def\Gc{{\mathcal{G}}}
\def\SRc{{\mathcal{SR}}}
\def\Rc{{\mathcal{R}}}

\def\RR{{\mathbb{R}}}
\def\ZZ{{\mathbb{Z}}}
\def\NN{{\mathbb{N}}}
\def\PP{{\mathbb{P}}}
\def\KK{{\mathbb{K}}}

\def\PP{{\mathbb{P}}}

\def\LL{{\mathbb{L}}}

\def\Spect{{\mathrm{Spect}}}
\def\dd{{\mathbf{d}}}

\journal{Journal of Algebra}

\begin{document}

\begin{frontmatter}



\title{On the total order of reducibility of a pencil of  algebraic plane curves}


\author[INRIA]{L. Bus\'e}
\ead{Laurent.Buse@inria.fr}
\ead[url]{http://www-sop.inria.fr/members/Laurent.Buse/}
\author[Toulouse]{G. Ch\`eze}
\ead{guillaume.cheze@math.univ-toulouse.fr}
\ead[url]{http://www.math.univ-toulouse.fr/$\sim$cheze/}

\address[INRIA]{Projet Galaad, INRIA Sophia Antipolis - M\'editerran\'ee\\
             2004 route des Lucioles, B.P. 93\\
             06902 Sophia Antipolis, France}
\address[Toulouse]{ Institut de Math\'ematiques de Toulouse\\
Universit\'e Paul Sabatier Toulouse 3 \\
MIP B\^at 1R3,\\
 31 062 Toulouse cedex 9, France}

\begin{abstract}
In this paper, the problem of bounding the number of reducible curves in a pencil of algebraic plane curves is addressed. Unlike most of the previous related works, each reducible curve of the pencil is here counted with its appropriate multiplicity. It is proved that this number of reducible curves, counted with multiplicity, is bounded by $d^2-1$ where $d$ is the degree of the pencil. Then, a sharper bound is given by taking into account the Newton's polygon of the pencil. 
\end{abstract}

\begin{keyword}
Pencil of algebraic curves \sep spectrum of a rational function \sep  algebraic de Rham's cohomology \sep Newton's polygon

\end{keyword}

\end{frontmatter}


\section*{Introduction}

Given a pencil of algebraic plane curves such that a general element is irreducible, the purpose of this paper is to give a sharp upper bound for the number of reducible curves in this pencil. This question has been widely studied in the literature, but never, as far as we know, by counting the reducible factors with their  multiplicities.

\medskip

Let $r(X,Y)=f(X,Y)/g(X,Y)$ be a rational function in $\KK(X,Y)$, where $\KK$ is an algebraically closed field. It is commonly said to be non-composite if it cannot be written $r=u \circ h$ where $h(X,Y) \in \KK(X,Y)$ and $u \in\KK(T)$ such that $\deg(u) \geq 2$ (recall that the degree of a rational function is the maximum of the degrees of its numerator and denominator after reduction). If $d=\max (\deg(f),\deg(g))$, we define
$$ f^\sharp(X,Y,Z)=Z^d f\left(\frac{X}{Z}, \frac{Y}{Z} \right), \ \ g^\sharp(X,Y,Z)=Z^d g\left(\frac{X}{Z}, \frac{Y}{Z} \right)$$
that are two homogeneous polynomials of the same degree $d$ in $\KK[X,Y,Z]$.
The set 
$$\sigma(f,g) =\{(\mu:\lambda) \in \PP^1_{\KK} \mid \mu f^\sharp + \lambda g^\sharp \textrm{ is reducible in } \KK[X,Y,Z] \} \subset \PP^1_{\KK}$$ 
is the spectrum of $r$ and a classical theorem of Bertini and Krull implies that it is finite if $r$ is non-composite. Actually, $\sigma(f,g)$ is finite if and only if $r$ is non-composite and if and only if the pencil of projective algebraic plane curves $\mu f^\sharp + \lambda g^\sharp =0$, $(\mu:\lambda) \in \PP^1_{\KK}$, has an irreducible general element (see for instance \cite[Chapitre 2, Th\'eor\`eme 3.4.6]{J79} and \cite[Theorem 2.2]{Bo} for detailed proofs). Notice that the study of $\sigma(f,g)$ is trivial if $d=1$. Therefore, throughout this paper we will always assume that $d\geq 2$.

Given $(\mu:\lambda) \in \sigma(f,g)$, a complete factorization of the polynomial \mbox{$\mu f^\sharp + \lambda g^\sharp$} is of the form
\begin{equation}
\mu f^\sharp + \lambda g^\sharp=\prod_{i=1}^{n(\mu:\lambda)} P_{(\mu:\lambda),i}^{e_{(\mu:\lambda),i}}	
\tag{$\star$}
\end{equation}
where each polynomial $P_{(\mu:\lambda),i}$ is irreducible and homogeneous in $\KK[X,Y,Z].$ If $\sigma(f,g)$ is finite the 
\emph{total order of reducibility}\footnote{This terminology is taken from \cite{St}.} $\rho(f,g)$ of $r$ is then defined by 
\begin{equation*}
\rho(f,g)=\sum_{(\mu:\lambda)\in \PP^1_\KK} \big( n(\mu:\lambda)-1 \big).	
\end{equation*}
Observe that the above sum is finite because $n(\mu:\lambda) \neq 1$ implies that \mbox{$(\mu:\lambda) \in \sigma(f,g)$}. 

\medskip

It is known that $\rho(f,g)$ is bounded above by $d^2-1$ where $d$ stands for the degree of $r$. As far as we know, the first related result has been given by Poincar\'e \cite{Poin}. He showed that 
$$ |\sigma(f,g)| \leq (2d-1)^2+2d+2.$$
This bound was improved only very recently by Ruppert \cite{Ru1} who proves that $|\sigma(f,g)|$ is bounded by $d^2-1$. This result was obtained as a byproduct of a very interesting technique developed by the author to decide the reducibility of an algebraic plane curve. 
Later on, Stein studied in \cite{St} a less general question but gave a stronger result: he proves that if $g=1$ then $\rho(f,1)\leq d-1$. Its approach, based on the study of the multiplicative group of all the divisors of the reducible curves in the pencil, is entirely different from that of Ruppert. Then, Stein's bound was improved in \cite{Kal} and after that several papers \cite{Lo,Vi,AHS,Bo} developed techniques with similar flavors to deal with the general case $\rho(f,g)$.  All of them obtained the bound $\rho(f,g)\leq d^2-1$ but also provide some various extensions: In \cite{Lo} the bound is proved in arbitrary characteristic, in \cite{Bo} it is shown that a direct generalization of Stein's result yields the bound $\rho(f,g)\leq d^2+d-1$, in \cite{Vi} the result is generalized to a very general ground variety and finally, in \cite{AHS} the authors were interested in a total reducibility order over a field $\KK$ that is not necessarily algebraically closed. Incidentally, point out the paper \cite{Yuz} that deals with completely reducible curves in a pencil, a topic which is closely related. 
 
\medskip

The aim of this paper is to study the total order of reducibility by counting  the multiplicities. More precisely, for each $(\mu:\lambda) \in \sigma(f,g)$  define 
$$m(\mu:\lambda):=\sum_{i=1}^{n(\mu:\lambda)} e_{(\mu:\lambda),i}$$ 
from the factorization $(\star)$.
This number is the number of factors of $\mu f^\sharp+ \lambda g^\sharp$ where the multiplicities of the  factors are counted. In particular, it is clear that $n(\mu:\lambda) \leq m(\mu:\lambda)$. 
We define the total order of reducibility with multiplicities of the rational function $r$ as the integer
$$m(f,g)=\sum_{(\mu:\lambda) \in \PP^1_\KK} \big(m(\mu:\lambda)-1 \big).$$
Obviously, it always holds that $0\leq \rho(f,g) \leq m(f,g)$. Moreover, notice that unlike $\rho(f,g)$,  $m(f,g)$ takes into account those curves in the pencil that are geometrically irreducible but scheme-theoretically non-reduced. However, it is proved in \cite[General Mixed Primset Theorem, p 74]{AHS} that the number of such curves is at most 4 in our context; we will come back to this point in Section \ref{mainresult}.

\medskip

The first main result of this paper is that the upper bound $d^2-1$ for $\rho(f,g)$ is also valid for $m(f,g)$. This is the content of Section \ref{mainresult} where it is assumed that the characteristic of $\KK$ is zero. Our method, which is inspired by \cite{Ru1}, is elementary compared to the previously mentioned papers. Roughly speaking, we will transform the pencil of curves into a pencil of matrices and obtain in this way the claimed bound as a consequence of rank computations of some matrices that we will study in Section  \ref{ruppert}. In this way, the known inequality $\rho(f,g)\leq d^2-1$ is easily obtained.  Moreover, we will actually not only bound $m(f,g)$ by  $d^2-1$, but a bigger quantity that takes into account the multiple factors of the reducible elements in the pencil. Notice that we will also show that the same bound holds in the case where  $r=f/g$ is a rational function in an arbitrary number of variables via a classical use of Bertini's Theorem at the end of Section \ref{mainresult}.

The second main result of this paper, given in Section \ref{newton}, is a refined upper bound for $m(f,g)$ which is obtained by considering the Newton's polygons of the polynomials $f$ and $g$. This result also gives a bound for the total order of reducibility $\rho(f,g)$ which is new and sharper. Notice that in this section the characteristic of $\KK$ will be assumed to be $0$ or $>d(d-1)$ where $d$ denotes the degree of $r=f/g$.

\subsection*{Notations}
Throughout this paper, $\KK$ stands for an algebraically closed field of cha\-ra\-cte\-ri\-stic $p$. 
Given a polynomial $f$, $\deg(f)$ denotes its total degree and 
$\partial_{X}f$ (resp.~$\partial_Y f$) denotes the partial derivative of $f$ with respect to the variable $X$ (resp.~to $Y$). 
Also, for any integer $n$ the notation $\KK[X,Y]_{\leq n}$ stands for the set of all the polynomials in $\KK[X,Y]$ with total degree less or equal to $n$; the notation $\KK[X,Y,Z]_n$ stands for the set of all homogeneous polynomials of degree $n$ in $\KK[X,Y,Z]$.

\section{Ruppert's linear map}\label{ruppert}

In the paper \cite{Ru1}, Ruppert introduced an original technique to decide whether a plane algebraic curve is reducible. Its formulation relies on the computation of the first de Rham's cohomology group of the complementary of the plane curve by means of linear algebra methods. Later, Gao followed this approach to obtain an algorithm for the factorization of a bivariate polynomial \cite{Ga}. 

\medskip

From now on, we will always assume in this section that the characteristic of the algebraically closed field $\KK$ is $p=0$.

\medskip

 For $\nu$ a positive integer and $f(X,Y) \in \KK[X,Y]$ a polynomial of degree $d \leq \nu$, define the $\KK$-linear map
\begin{eqnarray*}
\Gc_\nu(f): \KK[X,Y]_{\leq \nu-1} \times \KK[X,Y]_{\leq \nu-1}  & \longrightarrow & \KK[X,Y]_{\leq \nu+d-2}\\
(G,H) & \mapsto & 
f^2\left(
\partial_Y\left(\frac{G}{f}\right)-\partial_X\left(\frac{H}{f}\right)
\right) \\
 & & = \left|
 \begin{array}{cc}
 	f & \partial_Y f \\
 	G & \partial_Y G
 \end{array}
 \right| - 
 \left|
 \begin{array}{cc}
 	f & \partial_X f \\
 	H & \partial_X H
 \end{array}
 \right|.
\end{eqnarray*}
Let $f_1,\ldots,f_r$ be the irreducible factors of $f$. If $\gcd(f,\partial_X f)$ is a nonzero constant in $\KK$ then it is proved in \cite{Ga} that $\ker \Gc_d(f)$ is a $\KK$-vector space of dimension $r$ and that the set 
\begin{equation}\label{eq:gao_basis}
\left\{ \Big(\dfrac{f}{f_i}\partial_X f_i,\dfrac{f}{f_i}\partial_Y f_i\Big) \mid i=1,\dots,r\right\}
\end{equation}
is a basis of this kernel.
This result provides an explicit description of the kernel of the linear map $\Gc_d(f)$ if the polynomial $f(X,Y)$ does not have any square factor. In order to investigate this kernel in the general case, that is to say for an arbitrary polynomial $f \in \KK[X,Y]$ and for an arbitrary integer $\nu\geq \deg(f)$, we interpret it in terms of algebraic de Rham cohomology.

\medskip

Let $\nu$ be a positive integer and $0\neq f(X,Y) \in \KK[X,Y]$ be a polynomial of degree $d \leq \nu$. Assume that $f=f_1^{e_1}\cdots f_r^{e_r}$ is a factorization of $f$ where each polynomial $f_i$ is irreducible and denote by $\mathcal{C}$ the algebraic curve defined by the equation $f=0$. The first algebraic de Rham cohomology $H^1(\mathbb{A}^2_\KK \setminus \mathcal{C})$ is  the quotient of the closed 1-differential forms $w \in \Omega_{\KK[X,Y]_f / \KK}$ of $\KK[X,Y]_f$ over $\KK$ by the exact 1-forms. 

By definition of $\mathcal{G}_\nu(f)$, a couple $(G,H) \in \KK[X,Y]_{\leq \nu-1} \times \KK[X,Y]_{\leq \nu-1}$ belongs to the kernel of $\mathcal{G}_\nu(f)$ if and only if the 1-form $\frac{1}{f}(G\dd X+H\dd Y)$ is closed. Therefore, the kernel of $\mathcal{G}_\nu(f)$ is in correspondence with the closed 1-differential forms $w \in \Omega_{\KK[X,Y]_f / \KK}$ that can be written $w=\frac{1}{f}(G\dd X+H\dd Y)$ for some polynomials $G$ and $H$ of degree less or equal to $\nu-1$. As a consequence of Ruppert's results in \cite{Ru1} (see also \cite[Theorem 8.3]{Scheib}), these particular closed 1-forms are sufficient to give a representation of any element in $H^1(\mathbb{A}^2_\KK \setminus \mathcal{C})$, that is to say that the canonical map 
$$\ker \mathcal{G}_\nu(f) \rightarrow H^1(\mathbb{A}^2_\KK \setminus \mathcal{C})$$ 
is surjective. Actually, the closed 1-forms $\frac{\dd f_1}{f_1}, \ldots, \frac{\dd f_r}{f_r}$
are known to form a basis of 
$H^1(\mathbb{A}^2_\KK \setminus \mathcal{C})$ (see loc.~cit.~or for instance \cite[Chapter 6]{Dimca}). 
It follows that $$H^1(\mathbb{A}^2_\KK \setminus \mathcal{C}) \simeq \ker \mathcal{G}_\nu(f) / B_\nu$$ where $B_\nu$ is the set of 1-forms in $\ker \mathcal{G}_\nu(f)$ that are exact. Basically, the elements in $B_\nu$ are of the form $\dd\left(\frac{P}{f^s}\right)$ for some $P \in \KK[X,Y]$ and $s\in \NN$. However, we claim that the following equality holds
\begin{multline}\label{B}
B_\nu=\left\{ w=\frac{1}{f}(G\dd X+H\dd Y), (G,H) \in \KK[X,Y]_{\leq \nu-1} \times \KK[X,Y]_{\leq \nu-1} \right. \\
\left. \textrm{ such that } \exists P \in K[X,Y]_{\leq \nu} \textrm{ with } \dd \left(\frac{P}{f}\right)=w \right\}.
\end{multline}
It is a consequence of the following technical results.

\begin{Lem}\label{irrdiv}
	Let $p,q$ be polynomials in $\KK[X,Y]$ such that $p$ divides $q\dd p$. Then each irreducible factor of $p$ divides $q$.
\end{Lem}
\begin{proof} Let $p_1,\ldots,p_r$ be distinct irreducible factors of $p$ such that $p=\prod_{i=1}^rp_i^{e_i}$. Then the equality
$$\frac{\dd p}{p}=\sum_{i=1}^r e_i\frac{\dd p_i}{p_i}$$
together with our hypothesis 
imply that $p_i^{e_i}$ divides $q\sum_{j=1}^r e_j\frac{p}{p_j}\dd p_j$. We deduce that $p_i^{e_i}$ must divide $q\frac{p}{p_i} \dd p_i$ and therefore that $p_i$ divides $q$.
\end{proof}

\begin{Lem} Let $f\in \KK[X,Y]$ of degree $d$ and $G,H \in \KK[X,Y]$ of degree $\leq \nu-1$ with $\nu\geq d$. If $P \in \KK[X,Y]$ and $s\in \NN$ are such that 
	$$\dd\left(\frac{P}{f^s} \right)=\frac{1}{f}(G\dd X+H\dd Y)$$
	and $f$ does not divide $P$ if $s\geq 1$, then
	either $s=1$ and $\deg(P)\leq \nu$ or either $s=0$ and $\deg(P)\leq \nu-d$.
\end{Lem}
\begin{proof} This proof is inspired by \cite[Lemma 8.10]{Scheib}. Since
	$$\dd\left(\frac{P}{f^s} \right)=\frac{f\dd P-sP\dd f}{f^{s+1}}$$
we have
\begin{equation}\label{eq:Sch}
f\dd P-sP\dd f= f^s(G\dd X+H\dd Y).	
\end{equation}
Assume that $s\geq 2$ and denote by $f=\prod_{i=1}^r f_i^{e_i}$ an irreducible factorization of $f$. Equation \eqref{eq:Sch} implies that $f$ divides $P\dd f$ and therefore, by Lemma \ref{irrdiv}, that $f_i$ divides $P$ for all $i=1,\ldots,r$. Furthermore, since 
$$\frac{\dd f}{f}=\sum_{i=1}^r e_i\frac{\dd f_i}{f_i}$$ we get
$$f\dd P-sP\dd f=f\dd P - sPf\sum_{i=1}^r e_i\frac{\dd f_i}{f_i}=f(\dd P-s\sum_{i=1}^r e_i\frac{P}{f_i}\dd f_i).$$
But $f^s$ divides $f\dd P-sP\dd f$ by \eqref{eq:Sch}, so we deduce that
$$ f^{s-1} \, | \,  \dd P-s\sum_{i=1}^r e_i\frac{P}{f_i}\dd f_i.$$
Define $Q:=\gcd(f,P)=\prod_{i=1}^r f_i^{\mu_i}$ with $1\leq \mu_i\leq e_i$ for all $i=1,\ldots,r$ and set $R:=P/Q$. 
We obtain that $f^{s-1}$ divides
$$Q\dd R + R \dd Q -s\sum_{i=1}^r e_i\frac{P}{f_i}\dd f_i=Q\dd R + \sum_{i=1}^r (\mu_i - s e_i) R \frac{Q}{f_i}\dd f_i $$
since 
$$\frac{\dd Q}{Q}=\sum_{i=1}^r \mu_i\frac{\dd f_i}{f_i}.$$
As $s\geq 2$, $\mu_i-se_i<0$ for all $i$ and hence $f_i^{\mu_i}$ divides $R \frac{Q}{f_i}\dd f_i$. It follows that $f_i$ divides $R\dd f_i$ and therefore that $f_i$ divides $R$ by Lemma \ref{irrdiv}. But then $f_i^{\mu_i+1}$ divides $P$ which implies that $\mu_i=e_i$ for all $i$. Therefore, we conclude that if $s\geq 2$ then necessarily $f$ divides $P$: a contradiction with our hypotheses. So we must have $0\leq s\leq 1$.

Suppose that $s=0$. Then 
$$G\dd X + H\dd Y=f\dd P=f\partial_X P\dd X +f\partial_Y P\dd Y  $$
 and hence $\deg(P)\leq \nu - d$.

Now, assume that $s=1$.  We have
$$ fG \dd X + fH \dd Y = f\dd P - P\dd f = (f\partial_XP - P \partial_Xf)\dd X +(f\partial_YP - P \partial_Yf)\dd Y.$$
Denote by $\delta$ the degree of $P$ and by $P_\delta$, resp.~$f_d$, the homogeneous part of highest degree of $P$, resp.~$f$. 
If $f_d\partial_XP_\delta - P_\delta \partial_Xf_d \neq 0$ or $f_d\partial_YP_\delta - P_\delta \partial_Yf_d \neq 0$ then necessarily $\delta\leq \nu$ since $\deg(fG)\leq \nu + d -1$ and $\deg(fH)\leq \nu + d -1$. Otherwise, we obtain that 
$$\dd\left(\frac{P_\delta}{f_d}\right) = \frac{f_d\dd P_\delta-P_\delta\dd f_d}{(f_d)^2}=0$$
and hence that $\delta=d\leq \nu$.
\end{proof}

We are now ready to compute the dimension of the kernel of the $\KK$-linear map $\mathcal{G}_\nu(f)$ for all $\nu\geq d$.

\begin{Prop}\label{dimkerG} 
Let $f(X,Y) \in \KK[X,Y]$ of degree $d$ such that $f=f_1^{e_1}\cdots f_r^{e_r}$ is a factorization of $f$ where each polynomial $f_i$ is irreducible of degree $d_i$. Then, for all $\nu\geq d$ we have
$$\dim_\KK \ker \mathcal{G}_\nu (f) = r-1+\binom{2+\nu-d+\sum_{i=1}^r d_i(e_i-1)}{2}.$$
\end{Prop}
\begin{proof} From the above discussion on the interpretation of $\ker \mathcal{G}_\nu(f)$ in terms of 1-differential forms, we know that
	$$\dim_\KK \ker \mathcal{G}_\nu(f) = \dim_\KK H^1(\mathbb{A}^2_\KK \setminus \mathcal{C}) + \dim_\KK B_\nu$$
	where $B_\nu$ is defined by \eqref{B}. Since we also know that $\dim_\KK H^1(\mathbb{A}^2_\KK \setminus \mathcal{C})=r$, it remains to compute the dimension of $B_\nu$.
For that purpose, observe that the condition $\dd \left(\frac{P}{f}\right)=w$ in the definition of $B_\nu$ is equivalent to the system of equations
$$\begin{cases}
		f\partial_X P - P\partial_X f - Gf = 0 \\
		f\partial_Y P - P\partial_Y f - Hf = 0
\end{cases}$$
with the constraints $\deg(G)\leq \nu-1$, $\deg(H)\leq \nu-1$ and $\deg(P)\leq \nu$. 

Denote by $L_\nu$ the vector space of those triples $(G,H,P)$ solution of this system. The canonical projection $(G,H,P) \mapsto (G,H)$ sends $L_\nu$ to $B_\nu$. Moreover, the kernel of this projection are the triples $(0,0,P)$ satisfying the condition $\dd \left(\frac{P}{f}\right)=0$ which implies that $P$ is equal to $f$ up to multiplication by an element in $\KK$. Therefore, $\dim_\KK B_\nu=\dim_\KK L_\nu - 1$ and we are left with the computation of the dimension of $L_\nu$.

 The first equation defining $L_\nu$, that can be rewritten as
$f(\partial_X P -G)=P\partial_X f$, implies that $P$ must be of the form 
$$P=Q_1\frac{f}{\gcd(f,\partial_X f)}$$ where $Q_1$ is a polynomial of degree less or equal to $\nu-d+\deg \gcd(f,\partial_X f)$. Moreover, any such polynomial $P$  provides a couple $(P,G)$ that is solution of the above equation -- once $P$ is fixed then so does for $G$. A similar reasoning with the second defining equation of $L_\nu$ shows that its solutions are in correspondence with the polynomials $P$ of the form $Q_2f/\gcd(f,\partial_Y f)$ where  $Q_2$ is any polynomial of degree less or equal to $\nu-d+\deg \gcd(f,\partial_Y f)$. 

Now, to obtain the common solutions of the two defining equations of $L_\nu$ we have to solve the equation 
$$Q_2\gcd(f,\partial_X f)=Q_1\gcd(f,\partial_Y f).$$
But again, with similar arguments and using the fact that
$$\gcd\left(\gcd(f,\partial_X f),\gcd(f,\partial_Y f)\right)=\gcd(f,\partial_X f,\partial_Y f)$$
we get that  
$$Q_1=Q\frac{\gcd(f,\partial_X f)}{\gcd(f,\partial_X f,\partial_Y f)}, \ \ Q_2=Q\frac{\gcd(f,\partial_Y f)}{\gcd(f,\partial_X f,\partial_Y f)}$$
where $Q$ is any polynomial in $\KK[X,Y]$ of degree less or equal to
\begin{equation}\label{degQ}
	\nu-d+\deg \gcd(f,\partial_X f,\partial_Y f) = \nu-d+ \sum_{i=1}^r d_i(e_i-1).	
\end{equation}
Therefore, we deduce that the dimension of $L_\nu$ is equal to the dimension of the $\KK$-vector space of polynomials in $\KK[X,Y]$ of degree less or equal to the quantity \eqref{degQ}, that is to say
$$\binom{2+\nu-d+\sum_{i=1}^r d_i(e_i-1)}{2}$$
and the claimed formula is proved.
\end{proof}

Following Ruppert's approach in \cite{Ru1}, we introduce a new $\KK$-linear map which is similar to $\mathcal{G}_\nu(f)$ but with a source of smaller dimension. This property will be very important in the next section. To be more precise, for all positive integer $\nu$ consider the $\KK$-vector space
$$E_\nu=\{(G,H) \in \KK[X,Y]_{\leq \nu-1}\times \KK[X,Y]_{\leq \nu-1}  \text{ such that } \deg(XG+YH)\leq \nu-1\}.$$
It is of dimension $\nu^2-1$ and has the following property. 
\begin{Lem}\label{defrup} Let $f \in \KK[X,Y]$ of degree $d$. For all positive integer $\nu$ and all couple $(G,H) \in E_\nu$, the polynomial
		$$f^2\left(
		\partial_Y\left(\frac{G}{f}\right)-\partial_X\left(\frac{H}{f}\right) \right)$$
has degree at most $\nu+d-3$		
\end{Lem} 
\begin{proof}
	Denote by $G_{\nu-1}$, resp.~$H_{\nu-1}$, $f_d$, the homogeneous component of $G$, resp.~$H$, $f$ of degree $\nu-1$, resp.~$\nu-1$, $d$.  We have
		\begin{eqnarray*}
			f_d^2\, \dd \left(\frac{XG_{\nu-1}+YH_{\nu-1}}{f_d}\right)			& = & f_d \dd(XG_{\nu-1}+YH_{\nu-1})-(XG_{\nu-1}+YH_{\nu-1})\dd f_d \\
			& = & f_d(G_{\nu-1}+X\partial_XG_{\nu-1}+Y\partial_XH_{\nu-1})\dd X \\
			& & + f_d(H_{\nu-1}+X\partial_YG_{\nu-1}+Y\partial_YH_{\nu-1})\dd Y \\
			& & - (XG_{\nu-1}\partial_Xf_d+YH_{\nu-1}\partial_Xf_d)\dd X \\
&&- (YH_{\nu-1}\partial_Yf_d+XG_{\nu-1}\partial_Yf_d)\dd Y.
		\end{eqnarray*}
		So, using Euler's relation the coefficient of $\dd X$ is
		$$	f_d(\nu G_{\nu-1}-Y\partial_YG_{\nu-1}+Y\partial_XH_{\nu-1})-G_{\nu-1}(df_d-Y\partial_Yf_d)-YH_{\nu-1}\partial_Xf_d
		$$
		that is to say
		\begin{equation}\label{def1}
			(\nu-d)f_dG_{\nu-1} -Yf_d^2
			\left( 
			\partial_Y
			\left( \frac{G_{\nu-1}}{f_d}
			\right)
			-
			\partial_X
			\left( \frac{H_{\nu-1}}{f_d}
			\right)
			\right).	
		\end{equation}
		Similarly, the coefficient of $\dd Y$ is 
		\begin{equation}\label{def2}
			(\nu-d)f_dH_{\nu-1} -Xf_d^2
			\left( 
			\partial_Y
			\left( \frac{G_{\nu-1}}{f_d}
			\right)
			-
			\partial_X
			\left( \frac{H_{\nu-1}}{f_d}
			\right)
			\right).	
		\end{equation}
		
		Now, since $(G,H)\in E_\nu$ we have $XG_{\nu-1}+YH_{\nu-1}=0$. Therefore the quan\-ti\-ties \eqref{def1} and  \eqref{def2} are both equal to zero. It follows that
\begin{eqnarray*}
	0 &=& X\times \eqref{def1} + Y \times \eqref{def2} \\
&=&	(\nu-d)f_d(XG_{\nu-1}+YH_{\nu-1})-2XYf_d^2
	\left( 
	\partial_Y
	\left( \frac{G_{\nu-1}}{f_d}
	\right)
	-
	\partial_X
	\left( \frac{H_{\nu-1}}{f_d}
	\right)
	\right) \\
&=&
	-2XYf_d^2
	\left( 
	\partial_Y
	\left( \frac{G_{\nu-1}}{f_d}
	\right)
	-
	\partial_X
	\left( \frac{H_{\nu-1}}{f_d}
	\right)
	\right) 
\end{eqnarray*}
and the lemma is proved.		
\end{proof}

	Let $f(X,Y) \in \KK[X,Y]$ of degree $d$. For all integer $\nu\geq d$ we define the $\KK$-linear map
	$$\Rc_\nu(f) : E_{\nu}   \longrightarrow  \KK[X,Y]_{\leq \nu+d-3} : 
	(G,H) \mapsto f^2\left(
	\partial_Y\left(\frac{G}{f}\right)-\partial_X\left(\frac{H}{f}\right)
	\right).
	$$	
Point out that the operator $\Rc_\nu(-)$ is $\KK$-linear, that is to say that for all couples $(f,g) \in \KK[X,Y]_{\leq \nu}$ and all couple $(u,v)\in \KK^2$, we have $$\Rc_\nu(uf+vg)=u\Rc_\nu(f)+v\Rc_\nu(g).$$ Of course, a similar property holds for the operator $\Gc_\nu(f)$.

\begin{Prop}\label{dimprop}
	Let $f(X,Y) \in \KK[X,Y]$ of degree $d$. Then
	$$ \dim_\KK \ker \mathcal{R}_d(f) = \dim_\KK \ker \mathcal{G}_d(f) -1$$
and for all $\nu> d$
$$ \dim_\KK \ker \mathcal{R}_{\nu}(f) = \dim_\KK \ker \mathcal{G}_{\nu-1}(f).$$
\end{Prop}

\begin{proof} Denote by $G_{\nu-1}$, resp.~$H_{\nu-1}$, $f_d$, the homogeneous component of $G$, resp.~$H$, $f$ of degree $\nu-1$, resp.~$\nu-1$, $d$.

	First, notice that for all integer $\nu\geq d$ and all couple $(G,H) \in \ker \Gc_\nu(f)$ we have
		\begin{equation}\label{eq:key}
			\dd \left(\frac{XG_{\nu-1}+YH_{\nu-1}}{f_d}\right)=(\nu-d)\frac{G_{\nu-1}\dd X+H_{\nu-1}\dd Y}{f_d}.
		\end{equation}
Indeed, this follows from the computation we did in the proof of Lemma \ref{defrup}, more precisely the coefficients \eqref{def1} and \eqref{def2}.

	Now, let $f=f_1^{e_1}\cdots f_r^{e_r}$ be a factorization of $f$ where each polynomial $f_i$ is irreducible of degree $d_i$.
	By definition of both maps $\mathcal{G}_d(f)$ and  $\mathcal{R}_d(f)$, it is obvious to notice that any element in the kernel of $\mathcal{R}_d(f)$ is also in the kernel of $\mathcal{G}_d(f)$. Moreover, it is easy to check that 
	$$\left(\frac{f}{f_1}\partial_X f_1, \frac{f}{f_1}\partial_Y f_1\right) \in \ker \mathcal{G}_d(f)$$
but does not belong to the kernel of $\mathcal{R}_d(f)$ because
\begin{equation}\label{eq:euler}
X\frac{f}{f_1}\partial_X f_1+Y \frac{f}{f_1}\partial_Y f_1=\frac{f}{f_1}(X\partial_X f_1+Y\partial_Y f_1)=d_1f+\frac{f}{f_1}\tilde{f}_1
\end{equation}
where $\deg(\tilde{f}_1)<d_1$ (by Euler's relation). Nevertheless, 
for all couple $(G,H) \in \ker \mathcal{G}_d(f)$, Equation \eqref{eq:key} shows that there exists $\alpha \in \KK$ such that $$XG_{d-1}+YH_{d-1}=\alpha f_d.$$ It follows that  
$$(G,H)-\frac{\alpha}{d_1}\left(\frac{f}{f_1}\partial_X f_1, \frac{f}{f_1}\partial_Y f_1\right) \in \ker \mathcal{R}_d(f)$$
and therefore $$ \dim_\KK \ker \mathcal{R}_d(f) = \dim_\KK \ker \mathcal{G}_d(f) -1.$$

\medskip

To finish the proof, fix an integer $\nu>d$. It is clear from the definitions that 
$$\ker \mathcal{G}_{\nu-1}(f) \subseteq \ker \mathcal{R}_\nu(f) \subseteq \ker \mathcal{G}_{\nu}(f).$$
Pick a couple $(G,H) \in \ker \mathcal{R}_\nu(f)$. It satisfies $XG_{\nu-1}+YH_{\nu-1}=0$. Therefore, using \eqref{eq:key} we deduce that 
$$\frac{G_{\nu-1}\dd X+H_{\nu-1}\dd Y}{f_d}=0$$
that is to say that $G_{\nu-1}=H_{\nu-1}=0$. It follows that $(G,H) \in \ker \mathcal{G}_{\nu-1}(f)$.
\end{proof}

\begin{Cor}\label{DKrup}
	Let $f(X,Y) \in \KK[X,Y]$ of degree $d$ such that $f=f_1^{e_1}\cdots f_r^{e_r}$ is a factorization of $f$ where each polynomial $f_i$ is irreducible of degree $d_i$. Then
	$$\dim_\KK \ker \mathcal{R}_d(f) = r-2+\binom{2+\sum_{i=1}^r d_i(e_i-1)}{2}.$$
	In particular, $f(X,Y)$ is irreducible if and only if  $\dim_\KK \ker \mathcal{R}_d(f)=0$.
\end{Cor}

\begin{Rem}
If $f$ is a square-free polynomial, it is not hard to check that the set 
\begin{equation}\label{eq:basisfsq}
	\left\{ \left(-d_i \dfrac{f}{f_1}\partial_Xf_1+d_1 \dfrac{f}{f_i}\partial_Xf_i, -d_i \dfrac{f}{f_1}\partial_Yf_1+d_1 \dfrac{f}{f_i}\partial_Yf_i \right), \ i=2,\ldots,r
	\right\}
\end{equation}
	form a basis of the kernel of $\Rc_d(f)$. Indeed, Equation \eqref{eq:euler} implies that the elements of \eqref{eq:basisfsq} belongs to $E_d$. Furthermore, as already mentioned, the set \eqref{eq:gao_basis} form a basis of the kernel of $\Gc_d(f)$ when $f$ is square-free. It is then straightforward to check that the elements of \eqref{eq:basisfsq} are linearly independent over $\KK$ and then, using  Corollary \ref{DKrup}, to deduce that \eqref{eq:basisfsq} form a basis of the kernel of $\mathcal{R}_d(f)$.
\end{Rem}

\medskip

Since we will often deal with homogeneous polynomials in the rest of this paper, we need to extend  Corollary \ref{DKrup} to the case of  a homogeneous polynomial. To proceed, it is first necessary to define Ruppert's matrix in this setting. If $f(X,Y,Z) \in \KK[X,Y,Z]$ is a homogenous polynomial of degree $d$, we define
$$\Rc(f) : E  \longrightarrow  \KK[X,Y,Z]_{2d-3} : 
(G,H) \mapsto \frac{1}{Z}f^2\left(
\partial_Y\left(\frac{G}{f}\right)-\partial_X\left(\frac{H}{f}\right)
\right)
$$
where 
$$E=\{(G,H) \in \KK[X,Y,Z]_{d-1}\times \KK[X,Y,Z]_{d-1}  \text{ such that } Z | XG+YH\}.$$
Observe that the division by $Z$ in this definition is justified by Lemma \ref{defrup}. Here is the main result of this section.

\begin{Thm}\label{HDKrup}
	Let $f(X,Y,Z) \in \KK[X,Y,Z]$ homogeneous of degree $d$ and suppose that $f=f_1^{e_1}\cdots f_r^{e_r}$  where each polynomial $f_i(X,Y,Z)$ is irreducible and homogeneous of degree $d_i$. Then
	$$\dim_\KK \ker \mathcal{R}(f) = r-2+\binom{2+\sum_{i=1}^r d_i(e_i-1)}{2}.$$
	In particular, $f(X,Y,Z)$ is irreducible if and only if  $\dim_\KK \ker \mathcal{R}(f)=0$.
\end{Thm}
\begin{proof} Denote $\tilde{f}(X,Y)=f(X,Y,1) \in \KK[X,Y]$ and consider the map $\Rc_d(\tilde{f})$.
 We claim that the kernels of $\Rc(f)$ and $\Rc_d(\tilde{f})$ are isomorphic $\KK$-vector spaces. 

Indeed, let $(\tilde{G}(X,Y),\tilde{H}(X,Y)) \in \ker \Rc_d(\tilde{f})$ and set 
$$G(X,Y,Z)=Z^{d-1}\tilde{G}\left(\frac{X}{Z},\frac{Y}{Z}\right), \ \ H(X,Y,Z)=Z^{d-1}\tilde{H}\left(\frac{X}{Z},\frac{Y}{Z}\right).$$
Multiplying by $Z^{2d-2}$ the equality
\begin{multline*}
\tilde{f}\left(\frac{X}{Z},\frac{Y}{Z}\right) \partial_Y \tilde{G} \left(\frac{X}{Z},\frac{Y}{Z}\right) - \tilde{G}\left(\frac{X}{Z},\frac{Y}{Z}\right)\partial_Y\tilde{f}\left(\frac{X}{Z},\frac{Y}{Z}\right) \\ 
- \tilde{f}\left(\frac{X}{Z},\frac{Y}{Z}\right)\partial_X\tilde{H}\left(\frac{X}{Z},\frac{Y}{Z}\right) + \tilde{H}\left(\frac{X}{Z},\frac{Y}{Z}\right)\partial_X\tilde{f}\left(\frac{X}{Z},\frac{Y}{Z}\right) = 0
\end{multline*}
we get
$$f\partial_YG-G\partial_Yf-f\partial_XH+H\partial_Xf=Z\Rc(f)(G,H)=0.$$
Moreover, since $\deg(X\tilde{G}+Y\tilde{H})\leq d-1$ we deduce that $Z$ divides $XG+YH$ and conclude that $(G,H)$ belongs to the kernel of $\Rc(f)$. Similarly, if $(G,H) \in \ker \Rc(f)$ then $(\tilde{G},\tilde{H})=(G(X,Y,1),H(X,Y,1)) \in \ker \Rc_d(\tilde{f})$. Therefore, we have proved that
$$\dim_\KK \ker \mathcal{R}(f)=\dim_\KK \ker \mathcal{R}_d(\tilde{f}).$$
From here, if $\deg(\tilde{f})=d$ then the claimed equality follows from Corollary \ref{DKrup}. Now, if
$\deg(\tilde{f})<d$ then, by Proposition \ref{dimprop}, $\dim_{\KK} \ker \mathcal{R}_d(\tilde{f})=\dim_{\KK} \ker \mathcal{G}_{d-1}(\tilde{f})$. As $\deg(\tilde{f}) <d$, we can suppose that $f_r(X,Y,Z)=Z^{e_r}$, $d_r=1$, and then $\tilde{f}(X,Y)=f_1^{e_1}(X,Y,1)\cdots f_{r-1}^{e_{r-1}}(X,Y,1)$. Thus $\deg(\tilde{f})=d-e_r$ and $\tilde{f}$ has $r-1$ factors. Therefore, Proposition \ref{dimkerG} applied to $\tilde{f}$ yields the equality
$$ \dim_{\KK} \ker \mathcal{G}_{d-1}(\tilde{f})=(r-1)-1+ \binom{2+(d-1)-(d-e_r)+\sum_{i=1}^{r-1}d_i(e_i-1)}{2}$$
that gives the expected formula.
\end{proof}

\section{An upper bound for the total order of reducibility}\label{mainresult}

In this section, given a non-composite rational function \mbox{$r=f/g \in \KK(X,Y)$} we establish an upper bound for its total order of reducibility  counting multiplicities  $m(f,g)$ (recall that if $r$ is composite then $\sigma(f,g)$ is not a finite set). It turns out that this upper bound is the same as the known upper bound for the usual total order of reducibility $\rho(f,g)$ \cite{Lo,Vi}. Notice that we will actually prove a stronger result by considering a quantity which is bigger than $m(f,g)$. To proceed, we first need some notations. 

\medskip

Throughout this section, we will assume that the algebraically closed field $\KK$ has characteristic $p=0$.

\medskip

Given a non-composite rational function $r=f/g \in \KK(X,Y)$ of degree $d$, define the two homogeneous polynomials of degree $d$ in $\KK[X,Y,Z]$
$$ f^\sharp(X,Y,Z)=Z^d f\left(\frac{X}{Z}, \frac{Y}{Z} \right), \ \ g^\sharp(X,Y,Z)=Z^d g\left(\frac{X}{Z}, \frac{Y}{Z} \right).$$
If $(\mu:\lambda) \in \sigma(f,g)$ and 
\begin{equation}\label{genfactor}
	\mu f^\sharp(X,Y,Z) + \lambda g^\sharp(X,Y,Z)=\prod_{i=1}^{n(\mu:\lambda)} P_{(\mu:\lambda),i}^{e_{(\mu:\lambda),i}}
\end{equation}
where each polynomial $P_{(\mu:\lambda),i}$  is irreducible and homogeneous in $\KK[X,Y,Z]$, then 
$$\rho(f,g)=\sum_{(\mu:\lambda)\in \PP^1_\KK} \big( n(\mu:\lambda)-1 \big)$$
and 
$$m(f,g)=\sum_{(\mu:\lambda) \in \PP^1_\KK} 
\left( 
m(\mu:\lambda) 
-1 \right)
=\sum_{(\mu:\lambda) \in \PP^1_\KK} 
\left( 
\left(\sum_{i=1}^{n(\mu:\lambda)} e_{(\mu:\lambda),i}
\right) 
-1 \right).$$

The number of \emph{multiple factors} of $\mu f^\sharp(X,Y,Z) + \lambda g^\sharp(X,Y,Z)$, counted  with multiplicity, is 
$$\sum_{i=1}^{n(\mu:\lambda)} \left( e_{(\mu:\lambda),i} -1\right).$$
In the sequel we will actually balance each multiplicity in this sum with the degree of its corresponding factor, that is to say, we will rather consider the number
$$\omega(\mu:\lambda) = \sum_{i=1}^{n(\mu:\lambda)} \deg(P_{(\mu:\lambda),i}) \left( e_{(\mu:\lambda),i} -1\right) \geq \sum_{i=1}^{n(\mu:\lambda)} \left( e_{(\mu:\lambda),i} -1\right).$$
Consequently, we define   
$$\omega(f,g) = \sum_{(\mu:\lambda) \in \PP^1_\KK} 
\omega(\mu:\lambda). 
$$

Before going further in the notation, let us make a digression on the interesting quantity  $\omega(f,g)$ that first appears in the works of Darboux \cite{Dar} and Poincar\'e \cite{Poin} on the qualitative study of first order differential equations. In particular, they knew the following result:

\begin{Lem}\label{Ab}
Let $r=f/g \in \KK(X,Y)$ a non-composite reduced rational function of degree $d$. Then, 
$$\omega(f,g) \leq 2d-2.$$
\end{Lem}
\begin{proof}
	See \cite[Chapitre 2, Corollaire 3.5.6]{J79} for a detailed proof of this result valid with an arbitrary number of variables.
\end{proof}
It is also interesting to emphasize how Lemma \ref{Ab} implies that the cardinal of the set  
\begin{multline*}
	\gamma(f,g):=\left\{ (\mu:\lambda) \in \PP^1_{\KK}  \textrm{ such that } \mu f^\sharp + \lambda g^\sharp = {P}_{(\mu:\lambda)}^{e_{(\mu:\lambda)}} \right. \\
\left. \textrm{ with } e_{(\mu:\lambda)} \geq 2 \textrm{ and } P_{(\mu:\lambda)} \in \KK[X,Y,Z] \textrm{ irreducible } \right\} \subset \PP^1_{\KK}	
\end{multline*}
that is to say of the set of geometrically irreducible but reduced fibers\footnote{Notice that these fibers appear in the work of Poincar\'e \cite{Poin} as the \emph{critical remarkable values of fifth type}.}, is less or equal to 3.  
Indeed, Lemma \ref{Ab} yields  
$$ \sum_{(\mu:\lambda) \in \gamma(f,g)} \deg(P_{(\mu:\lambda)}) (e_{(\mu:\lambda)}-1) \leq 2d-2.$$
But obviously, $\deg(P_{(\mu:\lambda)})\leq \frac{d}{2}$ for all $(\mu:\lambda) \in \gamma(f,g)$ and, denoting by $|\gamma(f,g)|$ the cardinal of $\gamma(f,g)$, it follows that
\begin{multline*}
d\, |\gamma(f,g)| = \sum_{(\mu:\lambda) \in \gamma(f,g)} e_{(\mu:\lambda)}\deg(P_{(\mu:\lambda)}) \\ 
\leq 2d-2 + \sum_{(\mu:\lambda) \in \gamma(f,g)} \deg(P_{(\mu:\lambda)}) \leq 2d-2 + \frac{d}{2}\, |\gamma(f,g)|.	
\end{multline*}
Therefore, since $d$ is a positive integer we have $|\gamma(f,g)|\leq 3$. 

Mention that one can also be interested in fibers that are non reduced and geometrically irreducible on the affine space $\mathbb{A}^2_\KK$, say with variables $X,Y$, that is to say fibers of the pencil of curves $\mu f^\sharp + \lambda g^\sharp$ of the form $Z^{e_\infty}{P}^{e}$ where $P$ is an irreducible and homogeneous polynomial and $e\deg(P)+e_\infty=d$.   
Since there is at most one point $(\mu:\lambda) \in \PP^1_{\KK}$ such that $Z$ divides $\mu f^\sharp + \lambda g^\sharp$, we deduce from the inequality $|\gamma(f,g)|\leq 3$ that the number of such fibers is at most 4. This property actually appears in \cite[General Mixed Primset Theorem, p 74]{AHS}.

\medskip

Closing this parenthesis on the quantity $\omega(f,g)$, we finish with the notation by defining from \eqref{genfactor} the quantity
$$\theta(\mu,\lambda) = \binom{\omega(\mu:\lambda)+1}{2} - \sum_{i=1}^{n(\mu:\lambda)} (e_{(\mu:\lambda),i} -1) $$
which is positive since
$$\theta(\mu,\lambda) \geq \binom{\omega(\mu:\lambda)+1}{2} - \omega(\mu:\lambda)=\binom{\omega(\mu:\lambda)}{2}.$$
Finally, we set $$\theta(f,g) = \sum_{(\mu:\lambda) \in \PP^1_\KK} 
\theta(\mu:\lambda).$$

It is important to notice that we defined $\theta(\mu,\lambda)$ in order to have the equality
\begin{equation}\label{keyeq}
	m(\mu:\lambda)-1+ \omega(\mu:\lambda) + \theta(\mu:\lambda)= \dim \ker \Rc(\mu f^\sharp + \lambda g^\sharp)
\end{equation}
according to Theorem \ref{HDKrup}. 

\begin{Thm}\label{bound} 
	Let $r=f/g \in \KK(X,Y)$ a non-composite reduced rational function and set $d=\deg(r)=\max(\deg (f),\deg (g))$. We have
	$$ 0 \leq \rho(f,g) \leq m(f,g)+\omega(f,g)+\theta(f,g) \leq d^2-1.$$
\end{Thm}

\begin{proof} For all $(\mu:\lambda) \in \PP^1_\KK$, consider the linear map 
	$$\Rc(\mu f^\sharp +vg^\sharp)=\mu \Rc(f^\sharp)+v\Rc(g^\sharp)$$
	and its matrix
	$${\rm M}(\mu f^\sharp +vg^\sharp)=\mu {\rm M}(f^\sharp)+v{\rm M}(g^\sharp),$$
	where arbitrary bases for the $\KK$-vector spaces $E$ and $\KK[X,Y,Z]_{2d-3}$ have been chosen.
	
	They form a pencil of matrices that has $d^2-1$ columns and more rows.
	We define the polynomial $\Spect(U,V) \in \KK[U,V]$ as the greatest common divisor of all the $(d^2-1)$-minors of the matrix 
	\begin{equation}\label{Mfg}
	U {\rm M}(f^\sharp)+V {\rm M}(g^\sharp).	
	\end{equation}
	It is a homogeneous polynomial of degree $\leq d^2-1$, since each entry of \eqref{Mfg} is a linear form in $\KK[U,V]$.	
	

First, notice that $\Spect(U,V)$ is nonzero. Indeed, 
since $r=f/g$ is reduced and non-composite, the spectrum $\sigma(f,g)$ is finite and hence
there exists \mbox{$(\mu:\lambda) \notin  \sigma(f,g)$}. 
By Theorem \ref{HDKrup}, it follows that $\ker {\rm M}(\mu f^\sharp+ \lambda g^\sharp) =\{0\}$ and therefore that at least one of the $(d^2-1)$-minors of \eqref{Mfg} is nonzero since it has to be nonzero after the specializations of $U$ to $\mu$ and $V$ to $\lambda$.

Now, let $(\mu:\lambda) \in \sigma(f,g)$. By Theorem \ref{HDKrup}
\begin{equation}
\dim \ker {\rm M}(\mu f^\sharp+ \lambda g^\sharp) = m(\mu:\lambda)-1+ \omega(\mu:\lambda) + \theta(\mu:\lambda)>0.	
\end{equation}
Therefore, $(\mu:\lambda)$ is a root of $\Spect(U,V)$. Moreover, by a well-known property of characteristic polynomials, $(\mu:\lambda)$ is a root of $\Spect(U,V)$ of multiplicity at least 
$$  m(\mu:\lambda)-1+ \omega(\mu:\lambda) + \theta(\mu:\lambda).$$
 Summing all these multiplicities over all the elements in the spectrum $\sigma(f,g)$, we obtain the quantity $m(f,g)+\omega(f,g)+\theta(f,g)$. It is bounded above by $d^2-1$ because $\Spect(U,V)$ is a polynomial of degree less or equal to $d^2-1$.
\end{proof}

Observe that the term $m(f,g)+\omega(f,g)+\theta(f,g)$  depends quadra\-ti\-cal\-ly on the degrees and on the multiplicities of the irreducible components of the reducible curves in the pencil $\mu f^\sharp +\lambda g^\sharp$. This has to be compared with the bound $d^2-1$  which depends quadratically on the total degree $d$ of the pencil.

\medskip

As mentioned earlier, the inequality $\rho(f,g)\leq d^2-1$ has been proved in \cite{Lo,Vi}. This bound is known to be reached only for $d=1,2,3$ and several authors raised the question of the optimality of this bound for an arbitrary degree $d$ (see for instance \cite[Question 1, p.~79]{AHS} or \cite[top of p.~254]{Vi}).
Coming back to the total order of reducibility counting multiplicities, we do not know whether the bound $d^2-1$ given in Theorem \ref{bound} is optimal. Of course, it is optimal for $d=1,2,3$ since this is the case for the bound $\rho(f,g)\leq d^2-1$. Nevertheless, as a consequence of Theorem \ref{bound} we obtain the 
\begin{Cor}
Let $r=f/g \in \KK(X,Y)$ a non-composite reduced rational function of degree $d$. If $\rho(f,g)=d^2-1$ then $\omega(f,g)=0$. 
\end{Cor}
\noindent In other words, if there exists a pencil of curves with total order of reducibility equal to $d^2-1$ then it must have all its reducible members scheme-theoretically reduced. 

\medskip

 In the same spirit, given a polynomial $f\in \KK[X,Y]$ of degree $d$, one may ask if there exists a sharper bound for the spectrum $m(f):=m(f,1)$ than $d^2-1$. Indeed, as a consequence of a result of Stein \cite{St} (see also \cite{Lo} and \cite{AHS}), such a phenomenon appears when multiplicities of the irreducible factors are not considered; one has $\rho(f)\leq d-1$ (and this bound is reached). As pointed out to us by Dino Lorenzini, it turns out that the later inequality combined with Lemma \ref{Ab} implies that $m(f)\leq 3d-3$. 
  
The technique we used for proving Theorem \ref{bound} allows to show that 
\begin{equation}\label{PreStein}
	m(f) + \omega(f) + \theta(f) \leq d(d-1)/2
\end{equation}
providing $f$ is a non-composite polynomial. It follows from the fact that $\Rc_d(1)$ has rank $d(d-1)/2$, this rank being easy to compute since the linear map $\Rc_d(1)$ sends a couple $(G,H)$ to the difference $\partial_YG-\partial_XH$. We do not know if a bound linear in the degree $d$ holds for the quantity $m(f)+\omega(f) + \theta(f)$.

\medskip

 Although beyond the scope of this paper, we would like to mention that our approach can be directly applied for a collection of polynomials $(f_1,\ldots,f_r)$ rather than a couple of polynomials $(f,g)$. The problem is then to investigate the variety $\mathcal{S}$ of points $(\lambda_1,\ldots,\lambda_r)$ such that the polynomial  $\lambda_1f_1^\sharp+\cdots+\lambda_r f_r^\sharp$ is reducible, assuming that this latter is generically irreducible. As an immediate consequence of our approach, the degree of $\mathcal{S}$ is less or equal to $d^2-1$. Notice that the study of $\mathcal{S}$ has already been considered in \cite{BDN} in arbitrary characteristic.

\medskip

Finally, before closing this section we establish a result similar to Theorem \ref{bound} in the multivariate case. This kind of result is based on a classical use of Bertini's Theorem under the following form.  

\begin{Lem}\label{thm_bertini}
Let 
$$f= \sum_{|\underline e|\leq
d}c_{e_1, \ldots, e_n}X_{1}^{e_1} \ldots X_{n}^{e_n}\in \KK[X_{1},
\ldots ,X_{n}]$$ 
set $|\underline e|= e_1+ \cdots+ e_n$ and 
$$\LL:= \KK(U_1, \ldots, U_n, V_1, \ldots, V_n, W_1,
\ldots, W_n)$$ where $U_1, \ldots, U_n, V_1, \ldots, V_n, W_1,
\ldots, W_n$ are algebraically independent inde\-ter\-mi\-na\-tes. 

Then, the bivariate polynomial $$\tilde {f}(X,Y)= f(U_1X+V_1Y+ W_1,  \ldots,
U_nX+ V_nY+ W_n) \in \LL[X,Y]$$ is irreducible in $\overline{\LL}[X,Y]$ if
and only if $f$ is irreducible in $\KK[X_1,\dots,X_n]$.
\end{Lem}
\begin{proof} See \cite[lemma 7]{Ka}. See also \cite{Job} for a complete treatment of Bertini's Theorem.
\end{proof}

In the following theorem, the quantities $m(f,g)$, $\omega(f,g)$ and $\theta(f,g)$ that we have defined for a rational function $r=f/g$ in two variables are straight\-for\-war\-dly extended to a rational function in several variables, denoting by $X_0$ the homogenizing variable.

\begin{Thm}\label{bound_nvar}
Let $r=f/g \in \KK(X_1,\dots,X_n)$ a non-composite reduced ratio\-nal function of degree $d$. We have 
$$m(f,g) + \omega(f,g) + \theta(f,g) \leq d^2-1.$$
\end{Thm}

\begin{proof} Given $(\mu:\lambda) \in \PP^1_\KK$, Lemma \ref{thm_bertini} implies that 
$$\mu f^\sharp + \lambda g^\sharp=\prod_{i=1}^{n(\mu:\lambda)} P_{(\mu:\lambda),i}^{e_{(\mu:\lambda),i}}$$ 
with $P_{(\mu:\lambda),i}$ homogeneous and irreducible in $\KK[X_0,X_1,\dots,X_n]$, if and only if 
$$\mu \tilde{f}^\sharp + \lambda \tilde{g}^\sharp=\prod_{i=1}^{n(\mu:\lambda)} \tilde{P}_{(\mu:\lambda),i}^{e_{(\mu:\lambda),i}}$$ 
with $\tilde{P}_{(\mu:\lambda),i}$ homogeneous irreducible in $\overline{\LL}[X,Y,Z]$. Therefore,
$m(f,g)=m(\tilde{f},\tilde{g})$, $\omega(f,g)=\omega(\tilde{f},\tilde{g})$ and $\theta(f,g)=\theta(\tilde{f},\tilde{g})$. The claimed result then follows from Theorem \ref{bound} applied to the rational function $r=\tilde{f}/\tilde{g} \in \KK(X,Y)$.
\end{proof}

\section{Exploiting Newton's polygon}\label{newton}

In the previous section we considered rational functions $f/g$ with a certain fixed degree. In this section, we will refine this characterization by considering the Newton's polygons of $f$ and $g$. In this way, we will give an upper bound for the total order of reducibility counting multiplicities $m(f,g)$ that improves the one of Theorem \ref{bound} in many cases. In particular, an example for which this bound is almost reached for an arbitrary degree is presented.

To obtain this upper bound, we will follow a more basic approach than in Section \ref{mainresult}. Indeed, instead of using Theorem \ref{HDKrup} we will exhibit explicit elements in the kernel of a suitable Ruppert's linear map and show that they are linearly independent. This has the advantage to allow us working in non-zero characteristic, but has the disadvantage to provide a bound for the quantity $m(f,g)$ and not $m(f,g)+\omega(f,g)+\theta(f,g)$ as in Theorem \ref{HDKrup}. 

Before going further into details, mention that a bound for the total order of reducibility $\rho(f,g)$ related to the Newton's polygons of $f$ and $g$ is contained in the result of Vistoli \cite[Theorem 2.2]{Vi} since this amounts to homogenize the corresponding pencil of curves over a certain toric variety which is built from the Newton's polygons of $f$ and $g$. The bound provided in \cite{Vi} is then expressed in terms of invariants of this variety and of the pencil of curves that are not easy to make explicit.

\medskip

Recall that $p$ stands for the characteristic of the algebraically closed field $\KK$. We begin with some notations and preliminary materials.

\medskip

Given a polynomial $f(X,Y) \in \KK[X,Y]$, its support is the  set $S_f$ of integer points $(i,j)$ such that the monomial $X^iY^j$ appears in $f$ with a non zero coefficient. The convex hull, in the real space $\RR^2$, of $S_f$ is denoted by $N(f)$ and called the Newton's polygon of $f$. It is contained in the first quadrant of the plane $\RR^2$. \\
Recall that the Minkowski sum $A+B$ of two sets $A$ and $B \in \RR^2$ is the set of all elements $a+b$ with $a\in A$ and $b \in B$. We have the following classical result due to Ostrowski: let $f,f_1,\dots,f_r$ be polynomials in $\KK[X,Y]$ such that $f=f_1\ldots f_r$, then
\begin{equation}\label{Ostrowski} 
N(f)=N(f_1)+\cdots+N(f_r).
\end{equation}

Now, we introduce another polygon.

\begin{Def}
$N^+(f)$ is the smallest convex set that contains $N(f)$ and the origin, and that is bordered by edges having non-positive slopes (horizontal and vertical edges are hence allowed).
\end{Def}

Remark: An equivalent definition of $N^{+}(f)$ is the following: For any integer point $(i,j)\in \NN \times \NN $ we define its boxed Newton Polygon $B^+(i,j)$ to be all integer points in the rectangle with opposite corners $(0,0), (i,j)$. Then $N^+(f)$ is the convex hull of all $B^+(i,j)$ with $(i,j)$ in the support of $f$.\\

As $N(f)$, $N^+(f)$ is also contained in the first quadrant of the plane $\RR^2$. For example, $N^+(XY)$ is the square with vertices $(0,0)$, $(0,1)$, $(1,1)$ and $(1,0)$ and $N^+(X+Y+X^2Y^2)$ is the polygon with vertices $(0,0)$, $(1,0)$, $(2,2)$ and $(0,2)$. Moreover, it will be useful in the sequel to notice that if $f,g \in \KK[X,Y]$ are such that $g$ divides $f$ then clearly $N(f/g) \subset N^+(f)$. \\

The notion of total degree of a polynomial $f \in \KK[X,Y]$ can be refined in many ways in the sparse context. For instance, if $f(X,Y)=\sum_{i,j} f_{i,j} X^iY^j$ in $\KK[X,Y]$, given a couple $(a,b) \in \ZZ^2$ the $(a,b)$-weighted degree, or simply weighted degree, of $f$ is defined by 
$$d_{a,b}(f)= \max_{(i,j)\in \NN^2} \{ ai+bj \mid f_{i,j} \neq 0 \}.$$
Thus, the total degree of a polynomial $f$ is nothing but $d_{1,1}(f)$ and the degree of $f$ with respect to the variable $X$, resp.~$Y$,  corresponds to $\deg_{1,0}(f)$, resp.~$\deg_{0,1}(f)$. 

If $\mathcal{E}$ is an edge of a given convex set $\mathcal{N}$, denote by $a_{\mathcal{E}}X +b_{\mathcal{E}}Y=c_{\mathcal{E}}$ one of its integer equation. Then, it is clear that $d_{a_{\mathcal{E}},b_{\mathcal{E}}}(m)=d_{a_{\mathcal{E}},b_{\mathcal{E}}}(n)$ if $m, n \in \mathcal{E}$, and that 
$d_{a_{\mathcal{E}},b_{\mathcal{E}}}(m) \neq  d_{a_{\mathcal{E}},b_{\mathcal{E}}}(n)$
if $m \not \in \mathcal{E}$, $n \in \mathcal{E}$. In what follows we will use this remark for particular edges that we will call \emph{good edges}.

\begin{Def}\label{good} Suppose given a convex set $\mathcal{N}$ in the first quadrant of the plane. An edge $\mathcal{E}$ of $\mathcal{N}$ is called a good edge if the two following conditions hold:
\begin{itemize}
\item there exists $(a_{\mathcal{E}},b_{\mathcal{E}}) \in \NN^2\setminus (0,0)$ and $c_{\mathcal{E}} \in \NN$ such that $a_{\mathcal{E}}X+b_{\mathcal{E}}Y=c_{\mathcal{E}}$ is an equation of $\mathcal{E}$,
\item if $n \in \mathcal{N}$ , $n \not \in \mathcal{E}$ and $m  \in \mathcal{E}$ then $0\leq d_{a_{\mathcal{E}},b_{\mathcal{E}}}(n)<d_{a_{\mathcal{E}},b_{\mathcal{E}}}(m)$.
\end{itemize}

\end{Def}

Remarks: A good edge is a vertical edge or an edge with a non positive slope such that the convex set is below or to the left  of this edge.\\
A good edge does not always exist. Consider for example the triangle formed by $(1,0)$, $(2,2)$ and $(0,1)$.\\

We are now ready to state the main result of this section.

\begin{Thm}\label{bound_sparse}
Let $\mathcal{N}$ be a convex set in $\RR^2$. Denote by $\mathfrak{p}$ its number of  integral points and by 
$\mathfrak{p}_X$, resp.~$\mathfrak{p}_Y$, the number of points in $\mathcal{N}$ lying on the $X$-axis, resp.~$Y$-axis.
If $\mathcal{N}$ possesses a good edge $\mathcal{E}$, then $\mathfrak{p}_{\mathcal{E}}$ stands for the number of integral points in $\mathcal{N}$ lying on $\mathcal{E}$; otherwise set $\mathfrak{p}_{\mathcal{E}}=0$.

\medskip

Suppose given a non-composite reduced rational function $r=f/g \in \KK(X,Y)$ of degree $d$, assume that $\mathcal{N} \subseteq N \left( (1+X+Y)^d \right)$ and that the characteristic $p$ of $\KK$ is such that $p=0$ or  $p>d(d-1)$. 

\begin{itemize}
	
    \item If $N(f)\subset \mathcal{N}$ and $N(g) \subset \mathcal{N}$ then
	\begin{equation}\label{theq1}
	\rho(f,g) \leq 2\mathfrak{p}-\mathfrak{p}_X -\mathfrak{p}_Y- \mathfrak{p}_{\mathcal{E}}+\kappa.	
	\end{equation}	
	
	\item If $N^+(f)\subset \mathcal{N}$ and $N^+(g) \subset \mathcal{N}$ then
	\begin{equation}\label{theq2}
		m(f,g) \leq 2\mathfrak{p}-\mathfrak{p}_X -\mathfrak{p}_Y - \mathfrak{p}_{\mathcal{E}}+\kappa.			\end{equation}
	
	\item If $N(f)\subset \mathcal{N}$, $N(g) \subset \mathcal{N}$ and $(-g(0,0):f(0,0)) \not \in \sigma(f,g)$ then
	\begin{equation}\label{theq3}
		m(f,g) \leq 2\mathfrak{p}-\mathfrak{p}_X -\mathfrak{p}_Y - \mathfrak{p}_{\mathcal{E}}+\kappa.			\end{equation}
		
\end{itemize}
where $\kappa=\max(e_\infty-1,0)$ with $e_\infty$ the multiplicity (possibly 0) of the line at infinity $\{ Z=0 \}$ in the pencil of curves $\mu f^{\sharp}+\lambda g^{\sharp}$.
\end{Thm}

Before proceeding with the proof of this theorem, we comment it and illustrate it through three examples.
First, consider the dense case which corresponds to the situation studied in Section \ref{mainresult}. Here, we have $$\mathcal{N}=N(f)=N^+(f)=N(g)=N^+(g)=N((1+X+Y)^d)$$ 
and therefore $\mathfrak{p}=(d+2)(d+1)/2$, $\mathfrak{p}_X=\mathfrak{p}_Y=\mathfrak{p}_{\mathcal{E}}=d+1$, the good edge $\mathcal{E}$ being the diagonal joining the vertices $(0,d)$ and $(d,0)$. Moreover, since we are in the dense case, a linear change of coordinates leaves invariant $N(f)$, $N(g)$,  $\rho(f,g)$ and $m(f,g)$. Thus  we can assume that $\kappa=0$, that is to say that the line at infinity is not a factor of any member of pencil of curves $\mu f^{\sharp}+\lambda g^{\sharp}$. It follows that we obtain the expected bounds $\rho(f,g)\leq d^2-1$ and $m(f,g)\leq d^2-1$.

\medskip

Our next example, taken from \cite[Remark 5]{Lo}, is to show that the bound \eqref{theq2} is almost reached for an arbitrary degree $d$. 
Indeed, set  $$f(X,Y)=X(X+1)\cdots(X+d-2)Y+X, \ \ g(X,Y)=1.$$ 
It is not hard to check that $r=f/g$ is non-composite (see \cite[Remark 5]{Lo}) and  that $m(f,g)\geq 2d-2$. Now, defining $\mathcal{N}=N^+(f+g)$ which is a rectangle with vertices $(0,0)$, $(d-1,0)$, $(d-1,1)$ and $(0,1)$, we have $\mathfrak{p}=2d$, $\mathfrak{p}_X=d$, $\mathfrak{p}_Y=2$. Furthermore, we choose the horizontal good edge corresponding 
to $a_{\varepsilon}=0$, $b_{\varepsilon}=1$, $c_{\varepsilon}=1$ and obtain $\mathfrak{p}_{\mathcal{E}}=d$. Since
$\kappa=d-1$, the bound given in \eqref{theq2} is equal to $2d-1$ and we obtain 
$$ 2d-2\leq m(f,g) \leq 2d-1.$$

\medskip

Finally, our last example is to justify why we chose to state \eqref{theq3} despite the technical hypothesis requiring that the projective point \mbox{$(-g(0,0):f(0,0))$} does not belong to the spectrum of $f/g$. Consider the example
\begin{eqnarray*}
f(X,Y) & = & a_0+a_1XY+a_2X^2Y^2+a_3X^3Y^2+a_4X^2Y^3 \\ 
g(X,Y) & = & b_0+b_1XY+b_2X^2Y^2+b_3X^3Y^2+b_4X^2Y^3	
\end{eqnarray*}
where the coefficients $a_i$'s and $b_j$'s are all assumed to be nonzero and such that the above mentioned hypothesis is satisfied. We have $N(g)=N(f)$ and it is clear that $N(f)\subsetneq N^+(f)$. Taking $\mathcal{N}=N^+(f)$ and defining the good edge $\mathcal{E}$ as, for instance, the top horizontal edge of $N^+(f)$, we get $\mathfrak{p}=15$, $\mathfrak{p}_x=4$, $\mathfrak{p}_y=4$, $\mathfrak{p}_{\mathcal{E}}=3$, $d=5$. Therefore, \eqref{theq2}  yields
\begin{equation}\label{comp}
	m(f,g) \leq 2\mathfrak{p} -\mathfrak{p}_x-\mathfrak{p}_y- \mathfrak{p}_{\mathcal{E}}=19<d^2-1=24.
\end{equation} 
Now, choosing 
$\mathcal{N}=N(f)$ there is only one choice for the good edge $\mathcal{E}$ and we obtain  $\mathfrak{p}=5$, $\mathfrak{p}_x=1$, $\mathfrak{p}_y=1$, $\mathfrak{p}_\mathcal{E}=2$, $d=5$. Consequently, \eqref{theq3} gives $m(f,g) \leq  10$, to be compared with \eqref{comp}. The following picture shows the different polytopes involved in this example.
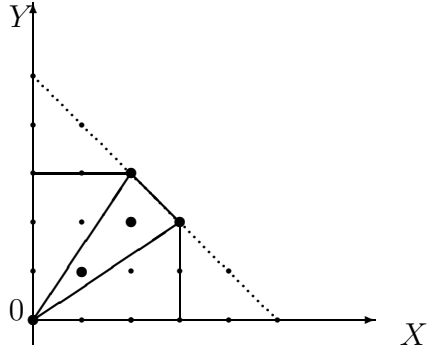
\begin{center}
\setlength{\unitlength}{.65cm}
\begin{picture}(10,10)
\put(0.5,2){$0$}
\put(1,2){\vector(1,0){7}}
\put(1,1.5){\vector(0,1){7}}
\put(0.5,8){$Y$}
\put(8.5,1.5){$X$}
\put(1,2){\circle*{.2}}
\put(2,3){\circle*{.2}}
\put(3,4){\circle*{.2}}
\put(4,4){\circle*{.2}}
\put(3,5){\circle*{.2}}
\put(3,2){\circle*{.1}}
\put(2,2){\circle*{.1}}
\put(4,2){\circle*{.1}}
\put(1,3){\circle*{.1}}
\put(1,4){\circle*{.1}}
\put(1,5){\circle*{.1}}
\put(1,6){\circle*{.1}}
\put(1,7){\circle*{.1}}
\put(2,4){\circle*{.1}}
\put(2,5){\circle*{.1}}
\put(2,6){\circle*{.1}}
\put(3,2){\circle*{.1}}
\put(3,3){\circle*{.1}}
\put(4,3){\circle*{.1}}
\put(5,3){\circle*{.1}}
\put(5,2){\circle*{.1}}
\put(6,2){\circle*{.1}}
\put(1,5){\line(1,0){2}}
\put(4,4){\line(0,-1){2}}
\multiput(1,7)(.1,-.1){50}{\circle*{.05}}
\thicklines\put(1,2){\line(3,2){3}}
\thicklines\put(1,2){\line(2,3){2}}
\thicklines\put(4,4){\line(-1,1){1}}
\end{picture}
\end{center}

We now turn to the proof of Theorem \ref{bound_sparse}. We begin with the following preliminary definition and result.
\begin{Def}\label{kerdef}
Let $f(X,Y) \in \KK[X,Y]$, let $(a_{\mathcal{E}},b_{\mathcal{E}}) \in \ZZ^2$ and let $f=f_1^{e_1}\cdots f_r^{e_r}$ be a factorization of $f$ where each polynomial $f_i$ is irreducible. For all $i=2,\dots,r$, we set
\begin{eqnarray*}
\mathcal{G}_i^{(1)} & = & -  d_{a_{\mathcal{E}},b_{\mathcal{E}}}(f_i)     \dfrac{f}{f_1}\partial_Xf_1+ d_{a_{\mathcal{E}},b_{\mathcal{E}}}(f_1)    \dfrac{f}{f_i}\partial_Xf_i \\
\mathcal{H}_i^{(1)} & = & - d_{a_{\mathcal{E}},b_{\mathcal{E}}}(f_i)    \dfrac{f}{f_1}\partial_Yf_1+  d_{a_{\mathcal{E}},b_{\mathcal{E}}}(f_1)     \dfrac{f}{f_i}\partial_Yf_i	
\end{eqnarray*}
and for all $i=1,\dots,r$ and $k=2,\dots,e_i$ we set
$$\mathcal{G}_i^{(k)}=\dfrac{f}{f_i^k}\partial_Xf_i, \  \ \mathcal{H}_i^{(k)}= \dfrac{f}{f_i^k}\partial_Yf_i.$$
\end{Def}

\begin{Prop}\label{inclusion}
	Let $f(X,Y) \in \KK[X,Y]$ be a polynomial of degree $d$, let $f=f_1^{e_1}\cdots f_r^{e_r}$ be a factorization of $f$ where each polynomial $f_i$ is irreducible and assume that the characteristic $p$ of $\KK$ is such that $p=0$ or $p>d$. 
\begin{itemize}
	
	\item[(i)] For all $i=2,\ldots,r$ and all $(a_{\mathcal{E}},b_{\mathcal{E}}) \in \ZZ^2$, $$N\Big(X \mathcal{G}_i^{(1)}\Big) \subset N(f) \textrm{ and } N\Big(Y\mathcal{H}_i^{(1)}\Big)\subset N(f) .$$

	\item[(ii)]  For all $i=1,\dots r$ and all $k=2,\dots, e_i$,
	 $$N\Big(X \mathcal{G}_i^{(k)}\Big) \subset N^+(f)\textrm{ and } N\Big(Y\mathcal{H}_i^{(k)}\Big)\subset N^+(f).$$
Furthermore, if $f(0,0)\neq 0$ then for all $i=1,\dots r$ and all $k=2,\dots, e_i$,
	 $$N\Big(X \mathcal{G}_i^{(k)}\Big) \subset N(f)\textrm{ and } N\Big(Y\mathcal{H}_i^{(k)}\Big)\subset N(f).$$

\item[(iii)] If ${\mathcal{E}}$ is a good edge of $N(f)$ with equation $a_{\mathcal{E}}X+b_{\mathcal{E}}Y=c_{\mathcal{E}}$, then for all $i=1,\dots r$ and all $k=1,\dots, e_i$, 
$$d_{a_{\mathcal{E}},b_{\mathcal{E}}}\Big(a_{\mathcal{E}} X \mathcal{G}_i^{(k)} + b_{\mathcal{E}} Y \mathcal{H}_i^{(k)}\Big) \leq d_{a_{\mathcal{E}},b_{\mathcal{E}}}(f)-1.$$

\item[(iv)] The  $(\sum_{i=1}^r e_i)-1$ elements
$$\left(
\mathcal{G}_i^{(k)},\mathcal{H}_i^{(k)}
\right), \ i=1,\ldots,r, \ k=1,\ldots,e_i, \ (i,k)\neq(1,1)$$ 
are $\KK$-linearly independent.

\end{itemize}
\end{Prop}

\begin{proof} We begin with the proof of (i) and (ii).\\
By Ostrowski's formula, see\eqref{Ostrowski}, $N\Big(X\dfrac{f}{f_i^k}\partial_Xf_i\Big)=N\Big(\dfrac{f}{f_i^k}\Big)+N(X\partial_Xf_i)$, 
and since $N(X \partial_Xf_i) \subset N(f_i)$ we get
\begin{equation}\label{eq1}
	N\Big(X\dfrac{f}{f_i^k}\partial_Xf_i\Big) \subset N\Big(\dfrac{f}{f_i^k}\Big)+N(f_i) = N\Big(\dfrac{f}{f_i^{k-1}}\Big)\subset N^+(f).
\end{equation}

If $k=1$,   Equation \eqref{eq1}, shows that $N\Big(X\dfrac{f}{f_i}\partial_Xf_i\Big) \subset N(f)$ for all $i=1,\ldots,r$ and hence that $N(X\mathcal{G}_i^{(1)})\subset N(f)$ for all $i=2,\ldots,r$.

If $k>1$ then, by \eqref{Ostrowski} we have $$N(f)=N\Big(\dfrac{f}{f_i^{k-1}}\Big)+N(f_i^{k-1}).$$ 
So, if $f(0,0) \neq 0$ then $f_i^{k-1}(0,0) \neq 0$ and hence $(0,0) \in N(f_i^{k-1})$. It follows that $$N\Big(\dfrac{f}{f_i^{k-1}}\Big) \subset N(f),$$ that proves that $N\Big(X \mathcal{G}_i^{(k)}\Big) \subset N(f)$ for all $i=1,\dots r$ and all $k=2,\dots, e_i$.

We can proceed similarly with the polynomials $h_i^{(k)}$ and conclude this way  the proof of (i) and (ii).

We turn to the proof of (iii). If $k>1$, then by Definition \ref{good} we clearly have 
$$d_{a_{\mathcal{E}},b_{\mathcal{E}}}(a_{\mathcal{E}} X \mathcal{G}_i^{(k)} + b_{\mathcal{E}} Y\mathcal{H}_i^{(k)}) \leq d_{a_{\mathcal{E}},b_{\mathcal{E}}}(f)-1.$$ 
If $k=1$, denote by $f^{\mathrm{top}}$ the homogeneous part of $f$ with maximum weighted degree $d_{a_{\mathcal{E}},b_{\mathcal{E}}}(f)$. Then, Euler's relation  
$$a_{\mathcal{E}}X \partial_X f^{\mathrm{top}} +b_{\mathcal{E}} Y \partial_Y f^{\mathrm{top}}= d_{a_{\mathcal{E}},b_{\mathcal{E}}}(f)  f^{\mathrm{top}}$$
allows to conclude.

It remains to prove (iv). 
For all $i=1,\ldots,r$ and $k=1,\ldots,e_i$, set
	\begin{equation}\label{eq:figi}
		g_i^{(k)}=\dfrac{f}{f_i^k}\partial_Xf_i, \ \ h_i^{(k)}=\dfrac{f}{f_i^k}\partial_Yf_i.
	\end{equation}
We will prove that the $\sum_{i=1}^r e_i$ elements $(g_i^{(k)},h_i^{(k)})$ defined by \eqref{eq:figi} are $\KK$-linearly independent and then the desired result will follow directly. So suppose that there exists a collection of $\lambda_{i,k} \in \KK$ such that
	$$\sum_{i=1}^r \sum_{k=1}^{e_i} \lambda_{i,k} g_i^{(k)}=0, \ \ \sum_{i=1}^r \sum_{k=1}^{e_i} \lambda_{i,k} h_i^{(k)}=0$$
	and choose an integer $j\in \{1,\ldots,r\}$.
	We have 
	$$\sum_{i=1}^r \sum_{k=1}^{e_i} \lambda_{i,k} g_i^{(k)} g_{j}^{(e_j)}=0$$ and since
	$f$ divides $g_i^{(k)}g_j^{(e_j)}$ for all $i\neq j$, we deduce that 
	$$\sum_{k=1}^{e_j} \lambda_{j,k} g_{j}^{(k)}g_j^{(e_j)} = 0 \mod f.$$
	Equivalently, there exists a polynomial $T \in \KK[X,Y]$ such that
	$$fT= \lambda_{j,1} \dfrac{f}{f_j}\dfrac{f}{f_j^{e_j}} (\partial_X f_j)^2+\cdots+\lambda_{j,e_j} \dfrac{f}{f_j^{e_j}}\dfrac{f}{f_j^{e_j}}(\partial_X f_j)^2$$
	that is to say, such that 
	$$f_j^{e_j}T=(\partial_X f_j)^2\dfrac{f}{f_j^{e_j}}(\lambda_{j,1} f_j^{e_j-1}+\cdots+\lambda_{j,e_j}).$$
	Therefore, either $f_j$ divides $\partial_X f_j$ or either $f_j^{e_j}$ divides \mbox{$(\lambda_{j,1} f_j^{e_j-1}+\cdots+\lambda_{j,e_j})$}.

	A similar reasoning by replacing  $g_i^{(k)}$ with $h_i^{(k)}$ shows that 
	either $f_j$ divides $\partial_Y f_j$ or either $f_j^{e_j}$ divides $(\lambda_{j,1} f_j^{e_j-1}+\cdots+\lambda_{j,e_j})$. Now, since $\deg f_j \leq d<p$ and $f$ is not a constant polynomial, $(\partial_X f_j, \partial_Y f_j) \neq (0,0)$ and hence $f_j$ cannot divide one of its partial derivative. It follows that necessarily $f_j^{e_j}$ divides $\lambda_{j,1} f_j^{e_j-1}+\cdots+\lambda_{j,e_j}$.
	But since $\deg (f_j^{e_j}) > \deg ( \lambda_{j,1} f_j^{e_j-1}+\cdots+\lambda_{j,e_j})$ we must have $\lambda_{j,1} f_j^{e_j-1}+\cdots+\lambda_{j,e_j}=0$. Furthermore $\deg f_j^{e_j}>\deg f_j^{e_j-1}>\cdots>\deg f_j$, so that $\lambda_{j,k}=0$ for all $j,k$. This proves that the $(g_i^{(k)},h_i^{(k)})$, hence the $(\mathcal{G}_i^{(k)},\mathcal{H}_i^{(k)})$,  are linearly independent over $\KK$.
\end{proof}

Observe that the technical hypothesis $f(0,0)\neq 0$ in (ii) is necessary. Indeed, if $f=X^3(Y^2+X+1)$, $f_1=X$ and $e_1=3$ then 
$X\mathcal{G}_1^{e_1}=X(Y^2+X+1)$ and its Newton's polygon is not included in the Newton's polygon  of $f$. 

\medskip

\noindent \emph{Proof of Theorem \ref{bound_sparse}}. We will proceed similarly to what we did to prove Theorem \ref{bound}. Given a polynomial $h(X,Y) \in \KK[X,Y]$, we define the $\KK$-linear map
	\begin{eqnarray*}
	\SRc(h) : E_{\mathcal{N}}  & \longrightarrow & \KK[X,Y]\\
	(G,H)& \longmapsto & h^2\left(
	\partial_Y\left(\frac{G}{h}\right)-\partial_X\left(\frac{H}{h}\right)
	\right)	
	\end{eqnarray*}
	where 
	\begin{multline*}
	E_{\mathcal{N}}=\{(G,H) \in \KK[X,Y]\times \KK[X,Y] \textrm{ such that } N(XG)\subset \mathcal{N}, N(YH)\subset \mathcal{N}, \\
	d_{a_{\mathcal{E}},b_{\mathcal{E}}}(a_{\mathcal{E}} X G + b_{\mathcal{E}} Y H) \leq d_{a_{\mathcal{E}},b_{\mathcal{E}}}(h)-1\}.
		\end{multline*}
Notice that the last condition in the above definition has to be forgotten if there does not exist a good edge $\mathcal{E}$. Also, observe that all the elements introduced in Definition \ref{kerdef} belong to the kernel of $\SRc(f)$, keeping the notation of loc.~cit.

Since we assumed that $\mathcal{N} \subseteq N \left( (1+X+Y)^d \right)$, $E_{\mathcal{N}}$ is a subvector space of the $\KK$-vector space $\KK[X,Y]_{\leq d-1}\times \KK[X,Y]_{\leq d-1}$, so that $\SRc(h)$ is a restriction of the $\KK$-linear map $\Gc_d(h)$ introduced in Section \ref{ruppert}, to $E_{\mathcal{N}}$. 
Let us compute the dimension of this latter vector space. Pick $(G,H) \in E_{\mathcal{N}}$. We have $N(XG) \subset \mathcal{N}$, hence $XG$ has at most $\mathfrak{p}-\mathfrak{p}_x$ nonzero coefficients and so does $G$, because $XG$ and $G$ have the same number of nonzero coefficients. Similarly, we get that $H$ has $\mathfrak{p}-\mathfrak{p}_y$ nonzero coefficients.  The condition  
$$d_{a_{\mathcal{E}},b_{\mathcal{E}}}(a_{\mathcal{E}} XG + b_{\mathcal{E}} Y H) \leq d_{a_{\mathcal{E}},b_{\mathcal{E}}}(f)-1$$ 
means that the weighted homogeneous part of highest degree of $G$ and $H$ are related. That is to say, we can write the homogeneous part of weighted degree $d_{a_{\mathcal{E}},b_{\mathcal{E}}}(f)-1$ of $H$ in terms of the homogeneous part of weighted degree $d_{a_{\mathcal{E}},b_{\mathcal{E}}}(f)-1$ of $G$. Consequently, we obtain 
\begin{equation}\label{dimEN}
	\dim_\KK E_{\mathcal{N}} = 2\mathfrak{p} -\mathfrak{p}_x -\mathfrak{p}_y -\mathfrak{p}_{\mathcal{E}}.
\end{equation}

Now, for all $(\mu:\lambda) \in \PP^1_\KK$, consider the linear map 
$$\SRc(\mu f+\lambda g)=\mu \SRc(f) + \lambda \SRc(g)$$
and, choosing bases for $E_{\mathcal{N}}$ and $\KK[X,Y]_{\leq 2d-2}$, the corresponding matrix
$${\rm M}(\mu f+\lambda g)=\mu {\rm M}(f) + \lambda {\rm M}(g).$$
They form a pencil of matrices that has $\dim_\KK E_{\mathcal{N}}$ columns and more rows. Then, define the  polynomial $\Spect(U,V)\in \KK[U,V]$ as the greatest common divisor of all the minors of size $\dim_\KK E_{\mathcal{N}}$ of the matrix
$$U{\rm M}(f)+V{\rm M}(g).$$
It is a homogeneous polynomial of degree lower or equal to $\dim_\KK E_{\mathcal{N}}$.

\medskip 

The polynomial $\Spect(U,V)$ is nonzero for the  same reason as the one given in Theorem \ref{bound}, since the linear maps $\SRc(-)$ are restrictions of the linear maps $\Gc_d(-)$. The fact that this property remains valid if $p>d(d-1)$ is a consequence of \cite[Lemma 2.4]{Ga} where Gao studied the property of the linear map $\Gc_d(-)$ for square-free polynomials in positive characteristic.

\medskip

Now, let $(\mu:\lambda) \in \sigma(f,g)$. Then $\dim \ker {\rm M}(\mu f+\lambda g) >0$ and   $(\mu:\lambda)$ is root of $\Spect(U,V)$ of multiplicity, say $\eta$. Since $\eta \geq \dim \ker {\rm M}(\mu f+\lambda g)$, Proposition \ref{inclusion} gives some lower bounds for $\eta$ that allow to conclude the proof of this theorem as follows. 
\begin{itemize}
	\item If $N(f), N(g) \subset \mathcal{N}$, then Proposition \ref{inclusion},(i),(iii) and (iv) implies that $\eta \geq n(\mu:\lambda) -1$ if $\deg(\mu f+\lambda g)=d$ or $\eta + \kappa \geq n(\mu:\lambda) -1$ otherwise. Summing over all the elements in $\sigma(f,g)$ we deduce the bound \eqref{theq1}.
	\item If $N(f)^+, N(g)^+ \subset \mathcal{N}$, then Proposition \ref{inclusion},(i), (ii) - first part, (iii) and (iv) implies that $\eta \geq m(\mu:\lambda) -1$ if $\deg(\mu f+\lambda g)=d$ or \mbox{$\eta + \kappa \geq m(\mu:\lambda) -1$} otherwise. Summing over all the elements in $\sigma(f,g)$ we deduce the bound \eqref{theq2}.
	\item If  $N(f), N(g) \subset \mathcal{N}$ and $(-g(0,0):f(0,0)) \notin \sigma(f,g)$, then Proposition \ref{inclusion},(i), (ii) - second part, (iii) and (iv) implies that $\eta \geq m(\mu:\lambda) -1$ if $\deg(\mu f+\lambda g)=d$ or $\eta + \kappa \geq m(\mu:\lambda) -1$ otherwise. Summing over all the elements in $\sigma(f,g)$ we deduce the bound \eqref{theq3}.
\end{itemize}
Notice that we used the fact that the polynomial $\Spect(U,V)$ is of degree lower or equal to $\dim_\KK E_{\mathcal{N}}$.
\hfill $\Box$

\medskip

To finish, point out that we can not state a result similar to Theorem \ref{bound_sparse} in terms of the Newton's polygon of $f$ and $g \in \KK[X_1,\dots,X_n]$ following the above strategy because we are not able to preserve the sparsity of the polynomials through Bertini's Theorem.

\bibliographystyle{plainnat}







\end{document}